\documentclass[onefignum,onetabnum]{siamart190516}

\usepackage[caption=false]{subfig}
\usepackage{listings}

\definecolor{codegreen}{rgb}{0,0.6,0}
\definecolor{codegray}{rgb}{0.5,0.5,0.5}
\definecolor{codepurple}{rgb}{0.58,0,0.82}

\newcommand*{\FormatDigit}[1]{\textcolor{magenta!80!black}{#1}}
\lstdefinestyle{CStyle}{
    backgroundcolor=\color{gray!10},   
    commentstyle=\color{codegreen},
    keywordstyle=\color{magenta!80!black},
    numberstyle=\tiny\color{codegray},
    stringstyle=\color{codepurple},
    basicstyle=\ttfamily\scriptsize,
    breakatwhitespace=false,         
    frame=single,
    stepnumber=5,
    breaklines=true,                 
    captionpos=b,                    
    keepspaces=true,                 
    numbers=left,                    
    numbersep=5pt,                  
    showspaces=false,                
    showstringspaces=false,
    showtabs=false,                  
    tabsize=2,
    language=C,
    morekeywords={Vec,Mat,KSP,IS,PetscBool,PetscViewer,MatPartitioning,PetscErrorCode},
    classoffset=1,
    morekeywords={PETSC_MAX_PATH_LEN,NULL,PETSC_COMM_WORLD,DIFFERENT_NONZERO_PATTERN,PETSC_ERR_USER,MAT_INITIAL_MATRIX,FILE_MODE_READ,MAT_COPY_VALUES},keywordstyle={\color{orange!50!black}},
    literate={0}{{\FormatDigit{0}}}{1}%
             {1}{{\FormatDigit{1}}}{1}%
             {2}{{\FormatDigit{2}}}{1}%
             {3}{{\FormatDigit{3}}}{1}%
             {4}{{\FormatDigit{4}}}{1}%
             {5}{{\FormatDigit{5}}}{1}%
             {6}{{\FormatDigit{6}}}{1}%
             {7}{{\FormatDigit{7}}}{1}%
             {8}{{\FormatDigit{8}}}{1}%
             {9}{{\FormatDigit{9}}}{1}%
             {.0}{{\FormatDigit{.0}}}{2}
             {.1}{{\FormatDigit{.1}}}{2}
             {.2}{{\FormatDigit{.2}}}{2}%
             {.3}{{\FormatDigit{.3}}}{2}%
             {.4}{{\FormatDigit{.4}}}{2}%
             {.5}{{\FormatDigit{.5}}}{2}%
             {.6}{{\FormatDigit{.6}}}{2}%
             {.7}{{\FormatDigit{.7}}}{2}%
             {.8}{{\FormatDigit{.8}}}{2}%
             {.9}{{\FormatDigit{.9}}}{2}%
             {\ }{{ }}{1}
             ,
}

\usepackage{lipsum}
\usepackage{amsfonts}
\usepackage[normalem]{ulem}
\usepackage{graphicx}
\usepackage{epstopdf}
\usepackage{pifont}
\usepackage{siunitx}
\usepackage{pgfplotstable}
    \usetikzlibrary{
        pgfplots.colorbrewer,
    }
    \pgfplotsset{
        cycle list/.define={my marks}{
            every mark/.append style={solid,fill=\pgfkeysvalueof{/pgfplots/mark list fill}},mark=*\\
            every mark/.append style={solid,fill=\pgfkeysvalueof{/pgfplots/mark list fill}},mark=square*\\
            every mark/.append style={solid,fill=\pgfkeysvalueof{/pgfplots/mark list fill}},mark=triangle*\\
            every mark/.append style={solid,fill=\pgfkeysvalueof{/pgfplots/mark list fill}},mark=diamond*\\
        },
    }
\usepackage{algpseudocode}
\usepackage{multirow}
\usepackage{adjustbox}
\usepackage{booktabs,colortbl}
\usepackage{cprotect}
\usepackage{stmaryrd}

\ifpdf
  \DeclareGraphicsExtensions{.eps,.pdf,.png,.jpg}
\else
  \DeclareGraphicsExtensions{.eps}
\fi

\newsiamremark{remark}{Remark}
\newsiamremark{hypothesis}{Hypothesis}
\crefname{hypothesis}{Hypothesis}{Hypotheses}
\newsiamthm{claim}{Claim}

\headers{Algebraic Preconditioner for Sparse Matrices}{H. Al Daas, P. Jolivet, and T. Rees}

\title{Efficient Algebraic Two-Level Schwarz Preconditioner For Sparse Matrices\thanks{Submitted to the editors \today.}}

\author{Hussam Al Daas\thanks{STFC Rutherford Appleton Laboratory, Harwell Campus, Didcot, Oxfordshire, OX11 0QX, UK 
  (\email{hussam.al-daas@stfc.ac.uk},
  \email{tyrone.rees@stfc.ac.uk}).}
  \and Pierre Jolivet\thanks{CNRS, ENSEEIHT, 2 rue Charles Camichel, 31071 Toulouse Cedex 7, France (\email{pierre.jolivet@enseeiht.fr}).}
\and Tyrone Rees\footnotemark[2] 
}

\usepackage{amsopn}
\DeclareMathOperator{\diag}{diag}

\newcommand{\C}{\mathbb{C}}
\newcommand{\hop}{H}

\newcommand{\Rr}{\mathcal{R}}

\newcommand{\part}[1]{\Omega_{#1}}
\newcommand{\res}[1]{R_{#1}}
\newcommand{\rest}[1]{R^\top_{#1}}
\newcommand{\resh}[1]{R^\hop_{#1}}

\newcommand{\nomega}[1]{n_{#1}}

\newcommand{\schwarz}[1]{M^{-1}_{\text{\tiny{#1}}}}

\pgfplotsset{compat=newest}
\pgfplotsset{colormap={paraview}{rgb(0cm)=(0.278431,0.278431,0.858824) rgb(0.1428571429cm)=(0,0,0.360784) rgb(0.2857142857cm)=(0,1,1) rgb(0.4285714286cm)=(0,0.501961,0) rgb(0.5714285714cm)=(1,1,0) rgb(0.7142857143cm)=(1,0.380392,0) rgb(0.8571428571cm)=(0.419608,0,0) rgb(1cm)=(0.878431,0.301961,0.301961)}}
\pgfplotsset{colormap={paraview-heat}{rgb(0cm)=(0.231373,0.298039,0.752941) rgb(0.5cm)=(0.865,0.865,0.865) rgb(1.0cm)=(0.705882,0.0156863,0.14902)}}
\usepackage{epstopdf}
\DeclareGraphicsExtensions{%
    .png,.PNG,%
    .pdf,.PDF,%
    .jpg,.mps,.jpeg,.jbig2,.jb2,.JPG,.JPEG,.JBIG2,.JB2}
\ifpdf
\hypersetup{
  pdftitle={Efficient Algebraic Two-Level Schwarz Preconditioner for Sparse Matrices},
  pdfauthor={H. Al Daas and P. Jolivet and T. Rees}
}
\fi

\begin{document}

\maketitle

\begin{abstract}
 Domain decomposition methods are among the most efficient for solving sparse linear systems of equations.
 Their effectiveness relies on a judiciously chosen coarse space. Originally introduced and theoretically
 proved to be efficient for self-adjoint operators, spectral coarse spaces have been proposed in the past
 few years for indefinite and non-self-adjoint operators. This paper presents a new spectral coarse space
 that can be constructed in a fully-algebraic way unlike most existing spectral coarse spaces.
 We present theoretical convergence result for Hermitian positive definite diagonally dominant matrices.
 Numerical experiments and comparisons against state-of-the-art preconditioners in the multigrid
 community show that the resulting two-level Schwarz preconditioner is efficient
 especially for non-self-adjoint operators. Furthermore, in this case, our proposed
 preconditioner outperforms state-of-the-art preconditioners.
\end{abstract}



\section{Introduction}
\label{sec:introduction}
In this paper, we develop an algebraic overlapping Schwarz preconditioner for 
the linear system of equations
\begin{equation*}
  \label{eq:Ax=b}
  Ax=b,
\end{equation*}
for a sparse matrix $A\in\C^{n\times n}$ and a given vector $b\in\C^n$.
Solving sparse linear systems of equations is omnipresent in scientific computing.
Direct approaches, based on Gaussian elimination,
have proved to be robust and efficient for a wide range of problems~\cite{DufER17}.
However, the memory required to apply sparse direct methods often scales poorly with the
problem size, particularly for three-dimensional discretizations of partial differential
equations (PDEs).
Furthermore, the algorithms underpinning sparse direct software are poorly suited to
parallel computation, which makes them difficult to adapt to emerging computing
architectures.

Iterative methods for solving linear systems~\cite{Saa19} have been an active research topic since early computers' days.
Their simple structure, at their most basic level requiring only matrix-vector multiplication and vector-vector operations, makes them attractive for tackling large-scale problems. However, since the convergence rate depends on the properties of the linear system, iterative methods are not, in general, robust.
For the class of iterative methods known as Krylov subspace methods,
we may alleviate this by applying a preconditioner, which transforms the problem
into one with more favorable numerical properties.  The choice of preconditioner is
usually problem-dependent, and a wide variety of preconditioning techniques have been proposed to improve the convergence rate of iterative methods,
see for example the recent survey~\cite{PeaP20} and the references therein.

Multilevel domain decomposition (DD) and multigrid methods are widely used preconditioners~\cite{DolJN15,SmiBG96,WidT05,XuZ17,Zha17}.
They have proved to be effective on a wide variety of matrices, but they are especially well suited to sparse Hermitian positive definite (HPD) matrices arising from
the discretization of PDEs.
Their efficiency stems from a judicious combination of a cheap fine-level solver with a coarse-space correction.
In the last two decades, there has been a great advance in the development of spectral coarse spaces that yield efficient preconditioners.
Spectral coarse spaces were initially proposed in the multigrid community for elliptic PDEs with self-adjoint 
operators~\cite{BreHMV99,ChaFHJMMRV03,EfeGV11,KolV06}, and similar ideas 
were later picked up by the DD community for the same kind of problems~\cite{AldG19,AldGJT19,AldJ21,BasSSS21,KlaRR15,KlaKR16,NatXDS11,SpiDHNPS14,SpiR13}.
The past three years have seen several approaches to tackle symmetric indefinite systems and non-self-adjoint problems.
For example, spectral coarse spaces for least squares problems and symmetric indefinite saddle-point systems were proposed in~\cite{AldJS21,NatT21}, where the problem is returned into the framework of self-adjoint operators.
An exciting new development is that a number of multigrid methods and spectral coarse spaces have been suggested for problems with indefinite or non-self-adjoint 
operators~\cite{BooD21,BooDGMS21,BooDGMS21_1,BooDJNOT21,BooDJT21,BooDNT20,DolJTCOR21,ManMRS19,ManRS18}. These coarse spaces are mainly based on heuristics and show efficiency on several challenging model problems arising from discretized PDEs.

A variety of mathematical tools such as the fictitious subspace lemma~\cite{Nep91} and local Hermitian positive semi-definite (HPSD) splitting matrices~\cite{AldG19} are now available to analyze and propose effective coarse spaces for self-adjoint operators.
However, these tools may not be directly used for indefinite or non-self-adjoint operators.
An alternative approach to studying the convergence of DD methods in the indefinite or non-self-adjoint case is to use Elman's theory~\cite{EisES83} of GMRES convergence, see for example \cite{BonCNT21,BooDGMS21,CaiW92}.

In this work, we propose a fully-algebraic spectral coarse space for the two-level Schwarz preconditioner for general sparse matrices.
We review the overlapping DD framework in~\cref{sec:DD}, including a summary of the main
features of local HPSD splitting matrices.
For each subdomain, we introduce the {\it local block splitting using lumping in the overlap} of a sparse matrix in~\cref{sec:twolevel}.
The coarse space is then constructed by solving locally and concurrently a generalized eigenvalue problem involving the local
block splitting matrix and the local subdomain matrix.
In the case where the matrix is HPD diagonally dominant, we prove that the local block splitting matrices are local HPSD splitting matrices, and in that case we show
that one can bound the condition number of the preconditioned matrix from above by a user-defined number.
Based on this heuristic, we generalize our approach for other cases.
Unlike most existing spectral coarse spaces, especially those suggested for indefinite or non-self-adjoint operators, we obtain the matrices involved
in the local generalized eigenvalue problem efficiently from the coefficient matrix; the preconditioner is therefore fully-algebraic.
In order to assess the proposed preconditioner, we provide in~\cref{sec:numerical_experiments} a set of numerical experiments on problems arising from a wide range of applications including convection-diffusion equation and other linear systems from the SuiteSparse Collection~\cite{DavH11}.
Furthermore, we compare our proposed preconditioner against state-of-the-art preconditioners in the multigrid community.
Finally, we give concluding remarks and future lines of research in \cref{sec:conclusion}.

\paragraph{Notation}
Let $1 \le n \le m$ and let 
$M \in \C^{m \times n}$ be a complex sparse matrix.
Let $\llbracket 1,p\rrbracket$ denote the set of the first $p$ positive integers, and 
let $S_1 \subset \llbracket 1, m \rrbracket$ and $S_2 \subset \llbracket 1, n\rrbracket$.
$M(S_1, :)$ is the submatrix of $M$ formed by the rows whose indices belong to
$S_1$ and $M(:, S_2)$ is the submatrix of $M$ formed by the columns whose indices belong to $S_2$.
$M(S_1,S_2)$ denotes the submatrix formed by taking the rows whose indices belong
to~$S_1$ and only retaining the columns whose indices belong to $S_2$.
$[S_1, S_2]$ means the concatenation of any two sets of integers $S_1$ and $S_2$, where
the order of the concatenation is important.
$I_n$ is the identity matrix of size $n$, 
the transpose matrix of $M$ is denoted $M^\top$, and the adjoint of $M$, denoted $M^\hop$, is the conjugate transpose of $M$, i.e., $M^\hop = \bar{M}^\top$.
$ker(M)$ and $range(M)$ denote the null space and the range of $M$, respectively.

\section{Domain decomposition}
\label{sec:DD}
Consider  ${\cal G}(A)$, the adjacency graph of the coefficient matrix in (\ref{eq:Ax=b}), and number its nodes, $V$, from $1$ to $n$. 
Using a graph partitioning algorithm, we split $V$ into $N \ll n$ \emph{nonoverlapping subdomains}, i.e., disjoint subsets $\part{Ii}$, $i \in\llbracket 1,N\rrbracket$, of size $\nomega{Ii}$. 
Let $\part{\Gamma i}$ be the subset, of size  $\nomega{\Gamma i}$, of nodes that are distance one in ${\cal G}(A)$ from 
the nodes in $\part{Ii}$,  $i \in\llbracket 1,N\rrbracket$.
We define the \emph{overlapping subdomain}, $\part{i}$, as $\part{i}=[\part{Ii}, \part{\Gamma i}]$,
with size $\nomega{i} = \nomega{\Gamma i} + \nomega{Ii}$.
The complement of $\part{i}$ in $\llbracket1,n\rrbracket$ is denoted by $\part{\text{c} i}$.

Associated with $\part{Ii}$ is a restriction (or projection) matrix 
$\res{Ii}\in \mathbb{R}^{n_{Ii} \times n}$ given by
$\res{Ii} = I_n(\part{Ii},:)$.
$\res{Ii}$ maps from the global domain to subdomain $\part{Ii}$. Its transpose
$\rest{Ii}$ is a prolongation matrix that maps from subdomain $\part{Ii}$ to the global domain.
Similarly, we define $\res{i}=I_n(\part{i},:)$ as the restriction operator to the overlapping subdomain $\part{i}$.

We define the {\em one-level Schwarz preconditioner} as
\begin{equation}
  \label{eq:additive_one_level}
  M_{\text{\tiny ASM}}^{-1} = \sum_{i=1}^N \rest{i} A_{ii}^{-1} \res{i},
\end{equation}
where we assume $A_{ii}=\res{i} A \rest{i}$ is nonsingular for $i\in\llbracket1,N\rrbracket$.

Applying this preconditioner to a vector involves solving concurrent local
problems in each subdomain.
Increasing $N$ reduces the size of the subdomains,
leading to smaller local problems and, correspondingly, faster computations.
However, in practice, preconditioning by $M_{\text{\tiny ASM}}^{-1}$ alone is often not be enough for convergence
of the iterative solver to be sufficiently rapid. 
We can improve convergence, while still maintaining robustness with respect to $N$, by applying a
suitably chosen \emph{coarse space}, or \emph{second-level}
 \cite{AldG19,BonCNT21,DolJN15,GanL17}. 

Let $\Rr\subset\C^n$ be a subspace of dimension $0 < n_0 \ll n$ and let $\res{0} \in \C^{n_0\times n}$
be a matrix such that the columns of $\resh{0}$ span the subspace $\Rr$. Assuming that $A_{00} = \res{0} A \resh{0}$ is nonsingular,
we define the {\em two-level Schwarz preconditioner} as
\begin{equation}
  \label{eq:additive_two_level_additive}
  M^{-1}_{\text{\tiny additive}} = \resh{0} A_{00}^{-1} \res{0} + M_{\text{\tiny ASM}}^{-1}.
\end{equation}
Such preconditioners have been used to solve a large class of systems arising from
a range of engineering applications (see, for example,
\cite{AldGJT19,AldJS21,HeiHK20,KonC17,MarCJNT20,SmiBG96,VanSG09} and references therein).

We denote by $D_i \in \mathbb{R}^{n_i\times n_i}$, $i \in\llbracket 1,N\rrbracket$, 
any non-negative diagonal matrices such that
\begin{equation*} 
  \sum_{i=1}^N \rest{i} D_{i} \res{i} = I_n.
\end{equation*}
We refer to $\left(D_i\right)_{1 \le i \le N}$ as an \emph{algebraic partition of unity}.

\paragraph{Variants of one- and two-level preconditioners}
The so-far presented Schwarz preconditioners are the additive one-level~\cref{eq:additive_one_level} and the additive two-level based on additive coarse space correction~\cref{eq:additive_two_level_additive}.
It was noticed in~\cite{CaiS99} that scaling the one-level Schwarz preconditioner by using the partition of unity yields faster convergence.
The resulting one-level preconditioner is referred to as restricted additive Schwarz and is defined as:
\begin{equation}
 \label{eq:restricted_one_level}
 M^{-1}_{\text{\tiny RAS}} = \sum_{i=1}^N \rest{i} D_{i} A_{ii}^{-1} \res{i}.
\end{equation}
Furthermore, there is a number of ways of how to combine the coarse space with a one-level preconditioner such as the additive, deflated, and balanced combinations, see for example~\cite{TanNVE09}.
Given a one-level preconditioner $M_{\star}^{-1}$, where the subscript $\star$ stands for either ASM or RAS, the two-level preconditioner with additive coarse space correction is defined as
\begin{equation*}
  M^{-1}_{\star,\text{\tiny additive}} = \resh{0} A_{00}^{-1} \res{0} + M_{\star}^{-1}.
\end{equation*}
The two-level preconditioner based on a deflated coarse space correction is
\begin{equation*}
 M^{-1}_{\star, \text{\tiny deflated}} = \resh{0} A_{00}^{-1} \res{0} + M_{\star}^{-1} \left(I-A\resh{0} A_{00}^{-1} \res{0}\right).
\end{equation*}
Due to its simple form, the additive two-level Schwarz based on the additive coarse space correction is the easiest to analyze.
However, we observe that the deflated variant combined with the restricted additive Schwarz preconditioner has better performance in practice.
The theory and presentation in this work employs the additive two-level Schwarz preconditioner using an additive coarse space correction, however, all numerical experiments involving the proposed preconditioner employ the 
restricted additive two-level Schwarz with deflated coarse space correction so that the two-level preconditioner used in \cref{sec:numerical_experiments} reads as
\begin{equation}
 \label{eq:restricted_two_level_deflated}
 M^{-1}_{\text{\tiny deflated}} = \resh{0} A_{00}^{-1} \res{0} + M_{\text{\tiny RAS}}^{-1} \left(I-A\resh{0} A_{00}^{-1} \res{0}\right).
\end{equation}
Note that the aforementioned variants are agnostic to the choice of the partitioning and the coarse space. That is,
once the restriction operators to the subdomains and the coarse space are set, all these variants are available.
\paragraph{Local HPSD splitting matrices of sparse HPD matrix}
A local HPSD matrix associated with subdomain $i$ is any HPSD matrix of the form
\[
 P_i \widetilde{A}_i P_i^\top = 
 \begin{pmatrix}
  A_{Ii} & A_{I\Gamma, i} &\\
  A_{\Gamma I,i} & \widetilde{A}_{\Gamma,i} & \\
  &&
  \end{pmatrix},
\]
where $P_i = I_n(\part{Ii},\part{\Gamma i},\part{\text{c} i},:)$ is a permutation matrix, $A_{Ii} = A(\part{Ii},\part{Ii})$, $A_{I\Gamma, i}=A_{\Gamma I,i}^\top = A(\part{Ii},\part{\Gamma i})$, and $\widetilde{A}_{\Gamma,i}$ is any HPSD
matrix such that the following inequality holds
\[
 0 \le u^\top \widetilde{A}_i u \le u^\top A u, \quad u \in \C^n.
\]

First presented and analyzed in~\cite{AldG19}, local HPSD splitting matrices provide a framework to construct robust two-level Schwarz preconditioners for sparse HPD matrices.  Recently, this has led to the introduction of robust multilevel Schwarz preconditioners for finite element SPD matrices~\cite{AldGJT19}, sparse normal equations matrices~\cite{AldJS21}, and sparse general SPD matrices~\cite{AldJ21}.

\section{Two-level Schwarz preconditioner for sparse matrices}
\label{sec:twolevel}
We present in this section the construction of a two-level preconditioner for sparse matrices.
First, we introduce a new local splitting matrix associated with subdomain $i$
that uses local values of $A$ to construct a preconditioner which is cheap to setup.
In the special chase where $A$ is HPD diagonally dominant, we prove that these local matrices are local HPSD splitting matrices of $A$.
We demonstrate that these matrices, when used to construct a GenEO-like coarse space,
produce a two-level Schwarz preconditioner that outperforms existing two-level preconditioners in many applications, particularly in the difficult case
where $A$ is not positive definite.

\subsection{Local block splitting matrices of $A$ using lumping in the overlap}
\label{sec:splitting}

\begin{definition}{Local block splitting matrix.}
  \label{def:splitting}
 Given the overlapping partitioning of $A$ presented in~\cref{sec:DD}, we have for each $i\in\llbracket1,N\rrbracket$
 \[
 P_i A P_i^\top = 
 \begin{pmatrix}
  A_{Ii} & A_{I\Gamma i} &\\
     A_{\Gamma I i} & A_{\Gamma i} & A_{\Gamma \text{c} i}\\
     &A_{c\Gamma i} & A_{\text{c}i}
 \end{pmatrix}
\]
    Let $s_i$ be the vector whose $j$th component is the sum of the absolute values of the $j$th row of $A_{\Gamma \text{c} i}$, and let $S_i = \diag(s_i)$.
Define $\widetilde{A}_{\Gamma i} = A_{\Gamma i} - S_i$.
 The {\it local block splitting matrix} of $A$ associated with subdomain $i$ is defined to be
 $\widetilde{A}_i = \rest{i} \widetilde{A}_{ii} \res{i}$, where
\[
 \widetilde{A}_{ii} = 
 \begin{pmatrix}
  A_{Ii} & A_{I\Gamma i}\\
  A_{\Gamma I i} & \widetilde{A}_{\Gamma i}
 \end{pmatrix}.
\]
\end{definition}

Note that we only require the sum of the absolute values of each row in the local matrix $A_{\Gamma \text{c} i}$ to construct the local block splitting matrix of $A$ associated with subdomain $i$. 
Then, each of these values is subtracted from the corresponding diagonal entry of the local matrix $A_{\Gamma i}$.
We can therefore construct $\widetilde{A}_{ii}$ cheaply and concurrently for each subdomain.

The following lemma shows that if $A$ is HPD diagonally dominant, the local splitting
matrices defined in~\cref{def:splitting} are local HPSD splitting matrices.

\begin{lemma}
 \label{lemma:hpsd_splitting}
Let $A$ be HPD diagonally dominant.
 The local block splitting matrix $\widetilde{A}_{i}$ defined in~\cref{def:splitting} is local HPSD splitting
 matrix of $A$ with respect to subdomain $i$.
\end{lemma}
\begin{proof}
 First, note that the $j$th diagonal element of $\widetilde{A}_{i}$ is
 \[
  \widetilde{A}_{i}(j,j) = \begin{cases}
    A(j,j)         & \text{ if } j \in \part{Ii},\\
    A(j,j)-s_i(j)  & \text{ if } j \in \part{\Gamma i},\\
      0              & \text{ if } j \in \part{\text{c} i},
   \end{cases}
 \]
    where $s_i$ is the vector whose $j$th component is the sum of the absolute values of the $j$th row of $A_{\Gamma \text{c} i}$.
 Therefore, by construction, $\widetilde{A}_{i}$ is Hermitian diagonally dominant, hence HPSD.
 Furthermore, $A-\widetilde{A}_{i}$ is Hermitian and diagonally dominant, hence HPSD.
 By the local structure of $\widetilde{A}_{i}$, we conclude it is HPSD splitting of $A$ with
 respect to subdomain $i$.
\end{proof}

\subsection{Coarse space}
\label{sec:coarse_space}
In this section we present a coarse space for the two-level Schwarz preconditioner.
For each $i\in\llbracket1,N\rrbracket$, given the local nonsingular matrix $A_{ii}$, the local splitting matrix $\widetilde{A}_{ii} = \res{i}\widetilde{A}_{i}\rest{i}$, and the partition of
unity matrix $D_i$, let $L_i=ker\left(D_iA_{ii}D_i\right)$ and $K_i = ker\left(\widetilde{A}_{ii}\right)$.
Now, define the following generalized eigenvalue problem:
\begin{align}
 \label{eq:gevp}
 \begin{split}
   &\text{find } (\lambda, u) \in \C\times \C^{n_i} \text{ such that}\\
   &\Pi_i D_i A_{ii} D_i \Pi_i u = \lambda \widetilde{A}_{ii} u,
 \end{split}
\end{align}
where $\Pi_i$ is the projection on $range\left(\widetilde{A}_{ii}\right)$.

Given a number $\tau > 0$, the coarse space we propose is defined to be the space generated by the columns of the matrix
\[
 \resh{i} = \begin{bmatrix} \rest{1}D_1Z_1&\cdots&\rest{N}D_NZ_N\end{bmatrix},
\]
where $Z_i$ is the matrix whose columns form a basis of the subspace
\begin{equation}
    \label{eq:tau}
 (L_i\cap K_i)^{\perp_{K_i}} \oplus span\left\{u \ | \ \Pi_i D_i A_{ii} D_i \Pi_i u = \lambda \widetilde{A}_{ii} u, |\lambda| > \frac{1}{\tau}\right\},
\end{equation}
where $(L_i\cap K_i)^{\perp_{K_i}}$ is the complementary subspace of $(L_i\cap K_i)$ inside $K_i$.

Note that in the case where $A$ is sparse HPD diagonally dominant, $\widetilde{A}_i$ is a local HPSD splitting matrix of $A$, and
the definition of the coarse space matches the one defined in~\cite{AldG19}.
Therefore, the two-level Schwarz preconditioner using the coarse space defined guarantees an upper bound on the condition number
of the preconditioned matrix
\[
 \kappa_2(M_{\text{\tiny ASM,additive}}^{-1} A) \le (k_c + 1) \left(2 + (2k_c +1)\frac{k_m}{\tau}\right),
\]
where $k_c$ is the number of colors required to color the graph of $A$ such that each two neighboring subdomains have different colors
and $k_m$ is the maximum number of overlapping subdomains sharing a row of $A$.
Therefore, when $A$ is sparse HPD diagonally dominant, the upper bound on $\kappa_2(M_{\text{\tiny ASM,additive}}^{-1} A)$ is independent of $N$ and can be controlled by
using the value $\tau$.

\section{Numerical experiments}
\label{sec:numerical_experiments}
In this section, we validate the effectiveness of the two-level method when
compared to other preconditioners.  \Cref{tab:suitesparse_comparison} presents
a comparison between four preconditioners: $\schwarz{\text{\tiny deflated}}$~\cref{eq:restricted_two_level_deflated},
 BoomerAMG~\cite{FalY02}, GAMG~\cite{AdaBKP04}, and
AGMG~\cite{Not10}.  The results are for the right-preconditioned GMRES~\cite{SaaS86} with a restart
parameter of 30 and a relative tolerance set to $10^{-8}$.  We highlight the
fact that our proposed preconditioner can handle unstructured systems, not
necessarily stemming from standard PDE discretization schemes, by displaying
some nonzero patterns in~\cref{fig:pattern}. For preconditioners used within
PETSc~\cite{PETSc} (all except AGMG), the systems are solved using 256 MPI processes. After
loading them from disk, their symmetric part $A^T + A$ is first renumbered by
ParMETIS~\cite{KarK98}. The resulting permutation is then applied to $A$ and
the corresponding linear systems are solved using a random right-hand side.
The initial guess is always zero.
The code that implements these steps is given in~\cref{sec:code}. For our DD
method, we leverage the PCHPDDM framework~\cite{JolRZ21} which is used to assemble
spectral coarse spaces using~\cref{eq:gevp}. The new option
\verb!-pc_hpddm_block_splitting!, introduced in PETSc 3.17, is used to
compute the local splitting matrices of $A$ from~\cref{sec:splitting}.
At most 60 eigenpairs are computed on each subdomain  and 
the threshold parameter $\tau$ from~\cref{eq:tau} is set to~0.3. 
These parameters were found to provide good numerical performance after a very
quick trial-and-error approach on a single problem.  We did not want to adjust
them for each problem individually, but it will be shown next that
they are fine overall without additional tuning.

Furthermore, a single subdomain is mapped to each process, i.e., $N = 256$
in~\cref{eq:restricted_one_level}.  Eventually, exact subdomain and second-level
operator LU factorizations are computed.

\pgfplotstableread{table.dat}\loadedtable
\begin{table}
  \caption{Preconditioner comparison.
    Iteration counts are reported. $M_{\text{\tiny deflated}}^{-1}$ is the restricted two-level
  overlapping Schwarz preconditioner as in~\cref{eq:restricted_two_level_deflated}.
   No value denotes iteration count exceeds \pgfmathprintnumber{100}.
   \dag~denotes either a failure in constructing the preconditioner or a breakdown in GMRES. 
   \ddag~denotes the problem is complex valued and the preconditioner is unavailable. Matrix identifiers that are emphasized correspond to symmetric matrices, otherwise matrices are non-self-adjoint.
   }
  \label{tab:suitesparse_comparison}
\pgfplotstablesort[sort key=m]{\sorted}{\loadedtable}
    \hspace*{-0.16\textwidth}
\begin{adjustbox}{width=1.15\textwidth}
\pgfplotstabletypeset[every head row/.style={before row=\hline,after row=\hline},
                      every last row/.style={after row=\hline},
                      every even row/.style={before row={\rowcolor[gray]{0.9}}},
                      columns/AGMG/.style={string type,string replace={-1}{},string replace={-2}{{\dag}}},
                      columns/HYPRE/.style={column name={BoomerAMG},string type,string replace={-3}{\ddag},string replace={-1}{},string replace={-2}{{\dag}}},
                      columns/GAMG/.style={string type,string replace={-3}{},string replace={-1}{},string replace={-2}{{\dag}}},
                      columns={identifier,m,nnz,AGMG,HYPRE,GAMG,HPDDM,nc},sort,sort key=m,
                      columns/m/.style={int detect,column name=$n$,dec sep align},
                      columns/nc/.style={int detect,column name=$n_0$,dec sep align},
                      columns/nnz/.style={int detect,column name=nnz($A$),dec sep align},
                      columns/HPDDM/.style={column name=$M_{\text{\tiny deflated}}^{-1}$},
                      every row 0 column 6/.style={postproc cell content/.style={@cell content=\textbf{##1}}},
                      every row 1 column 6/.style={postproc cell content/.style={@cell content=\textbf{##1}}},
                      every row 2 column 6/.style={postproc cell content/.style={@cell content=\textbf{##1}}},
                      every row 3 column 6/.style={postproc cell content/.style={@cell content=\textbf{##1}}},
                      every row 4 column 4/.style={postproc cell content/.style={@cell content=\textbf{##1}}},
                      every row 5 column 6/.style={postproc cell content/.style={@cell content=\textbf{##1}}},
                      every row 6 column 6/.style={postproc cell content/.style={@cell content=\textbf{##1}}},
                      every row 7 column 6/.style={postproc cell content/.style={@cell content=\textbf{##1}}},
                      every row 8 column 6/.style={postproc cell content/.style={@cell content=\textbf{##1}}},
                      every row 9 column 6/.style={postproc cell content/.style={@cell content=\textbf{##1}}},
                      every row 10 column 6/.style={postproc cell content/.style={@cell content=\textbf{##1}}},
                      every row 11 column 6/.style={postproc cell content/.style={@cell content=\textbf{##1}}},
                      every row 12 column 6/.style={postproc cell content/.style={@cell content=\textbf{##1}}},
                      every row 13 column 4/.style={postproc cell content/.style={@cell content=\textbf{##1}}},
                      every row 14 column 6/.style={postproc cell content/.style={@cell content=\textbf{##1}}},
                      every row 15 column 6/.style={postproc cell content/.style={@cell content=\textbf{##1}}},
                      every row 16 column 6/.style={postproc cell content/.style={@cell content=\textbf{##1}}},
                      every row 17 column 4/.style={postproc cell content/.style={@cell content=\textbf{##1}}},
                      every row 18 column 4/.style={postproc cell content/.style={@cell content=\textbf{##1}}},
                      columns/identifier/.style={string type},display columns/0/.style={column name=Identifier, column type={l|},
postproc cell content/.append code={%
\pgfplotstablegetelem{\pgfplotstablerow}{sym}\of{\sorted}
\ifnum\pgfplotsretval=1
\pgfkeyssetvalue{/pgfplots/table/@cell content}{\emph{##1}}%
\fi
}
                      }]\loadedtable
                      \end{adjustbox}
\end{table}
\begin{figure}
    \caption{Nonzero sparsity pattern of some of the test matrices from~\cref{tab:suitesparse_comparison}.\label{fig:pattern}}
    \begin{center}
    \subfloat[nxp1, $n = \pgfmathprintnumber{410004}$]{
\begin{tikzpicture}
    \begin{axis}[enlargelimits=false, axis on top, axis equal image, width=6.1cm,yticklabels={,,},xticklabels={,,}]
\addplot graphics [xmin=1,xmax=2,ymin=1,ymax=2] {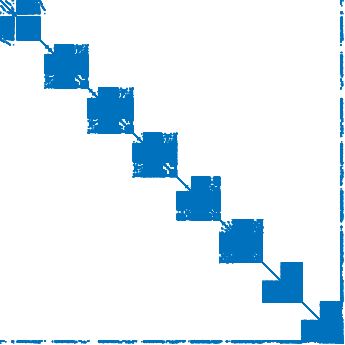};
\end{axis}
\end{tikzpicture}
    }
    \subfloat[CoupCons3D, $n = \pgfmathprintnumber{420000}$]{
\begin{tikzpicture}
    \begin{axis}[enlargelimits=false, axis on top, axis equal image, width=6.1cm,yticklabels={,,},xticklabels={,,}]
\addplot graphics [xmin=1,xmax=2,ymin=1,ymax=2] {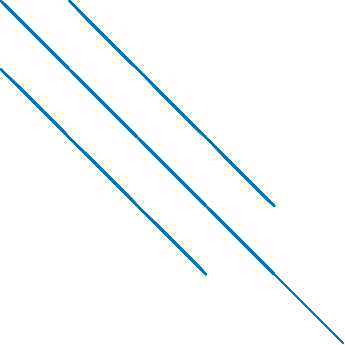};
\end{axis}
\end{tikzpicture}
    }
    \subfloat[memchip, $n = \pgfmathprintnumber{2700524}$]{
\begin{tikzpicture}
    \begin{axis}[enlargelimits=false, axis on top, axis equal image, width=6.1cm,yticklabels={,,},xticklabels={,,}]
\addplot graphics [xmin=1,xmax=2,ymin=1,ymax=2] {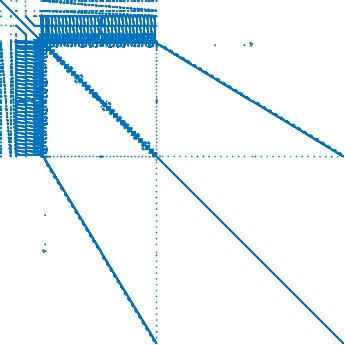};
\end{axis}
\end{tikzpicture}
    }
\end{center}
\end{figure}

A convection-diffusion problem will now be investigated. It reads:
\begin{align}
    \begin{split}
    \nabla \cdot (V u) - \nu \nabla\cdot(\kappa \nabla u) = 0 & \text{ in } \Omega \label{eq:conv-diff} \\
    u = 0 & \text{ in } \Gamma_0 \\ 
    u = 1 & \text{ in } \Gamma_1.
    \end{split}
\end{align}
The problem is SUPG-stabilized~\cite{BroH82} and discretized by FreeFEM~\cite{Hec12}. It is important to keep in mind that the proposed preconditioner is algebraic, thus there is no specific transfer of information from the discretization kernel to solver backend. The domain $\Omega$ is either the unit square or the unit cube meshed semi-structurally to account for boundary layers, see an example of such a mesh in~\cref{fig:mesh}. The value of $\nu$ is constant in $\Omega$. The value of $\kappa$ is given in~\cref{fig:kappa}. The value of the velocity field $V$ is either:
\begin{equation*}
    V(x,y) = \begin{pmatrix}x(1-x)(2y-1)\\-y(1-y)(2x-1)\end{pmatrix} \quad \text{ or } \quad V(x,y,z)=
        \begin{pmatrix}
2x(1-x)(2y-1)z\\
-y(1-y)(2x-1)\\
-z(1-z)(2x-1)(2y-1)
        \end{pmatrix} ,
\end{equation*}
in 2D and 3D, respectively. These are standard values taken from the literature~\cite{Not12}. The definition of $\Gamma_0$ and $\Gamma_1$ may be inferred by looking at the two- and three-dimensional solutions in~\cref{fig:2d-0,fig:2d-2,fig:2d-4} and~\cref{fig:3d-0,fig:3d-2,fig:3d-4}, respectively. The iteration counts reported in~\cref{tab:conv-diff} show that the proposed preconditioner handles this problem, even as $\nu$ tends to zero. In 2D, the operator, resp. grid, complexity is of at most 1.008, resp.\ 1.43. In 3D, these figures are 1.02 and 1.7, respectively.
\begin{table}
    \caption{Iteration counts of the proposed preconditioner for solving the two- and three-dimensional convection-diffusion problem from~\cref{eq:conv-diff} with order $k$ Lagrange finite element space. The number of subdomains is $N$ and the size of the discrete system is $n$. After each iteration count for each $\nu$, the size of second-level operator is typeset between parentheses.\label{tab:conv-diff}}
\begin{adjustbox}{width=\linewidth}
\begin{tabular}{c|c|c|c|cccccc}
  \hline
  \multirow{2}{*}{Dimension} &
  \multirow{2}{*}{$k$} & 
  \multirow{2}{*}{$N$} &
  \multirow{2}{*}{$n$} &
  \multicolumn{5}{c}{$\nu$}
      \\ 
    & & & & 1 & $10^{-1}$ & $10^{-2}$ & $10^{-3}$  & $10^{-4}$ \\  \hline
    2 & 1 & \pgfmathprintnumber{1024} & \pgfmathprintnumber[precision=1]{6255001} & 23 \tiny{(\pgfmathprintnumber{52875})} & 20 \tiny{(\pgfmathprintnumber{52872})} & 19 \tiny{(\pgfmathprintnumber{52759})} & 20 \tiny{(\pgfmathprintnumber{47497})} & 21 \tiny{(\pgfmathprintnumber{28235})} \\\hline
    3 & 2 & \pgfmathprintnumber{4096} & \pgfmathprintnumber[precision=1]{8120601} & 18 \tiny{(\pgfmathprintnumber[precision=1]{181796})} & 14 \tiny{(\pgfmathprintnumber[precision=1]{181342})} & 11 \tiny{(\pgfmathprintnumber[precision=1]{161413})} & 16 \tiny{(\pgfmathprintnumber{97657})} & 29 \tiny{(\pgfmathprintnumber{76853})} \\\hline
\end{tabular}
\end{adjustbox}
\end{table}
\begin{figure}
    \caption{(a) Mesh, (b) diffusivity coefficient, and solutions of some of the (c)--(e) two- and (f)--(h) three-dimensional test cases from~\cref{tab:conv-diff}.\label{fig:u}}
    \begin{center}
    \subfloat[Semi-structured mesh\label{fig:mesh}]{
        \includegraphics[width=0.25\textwidth]{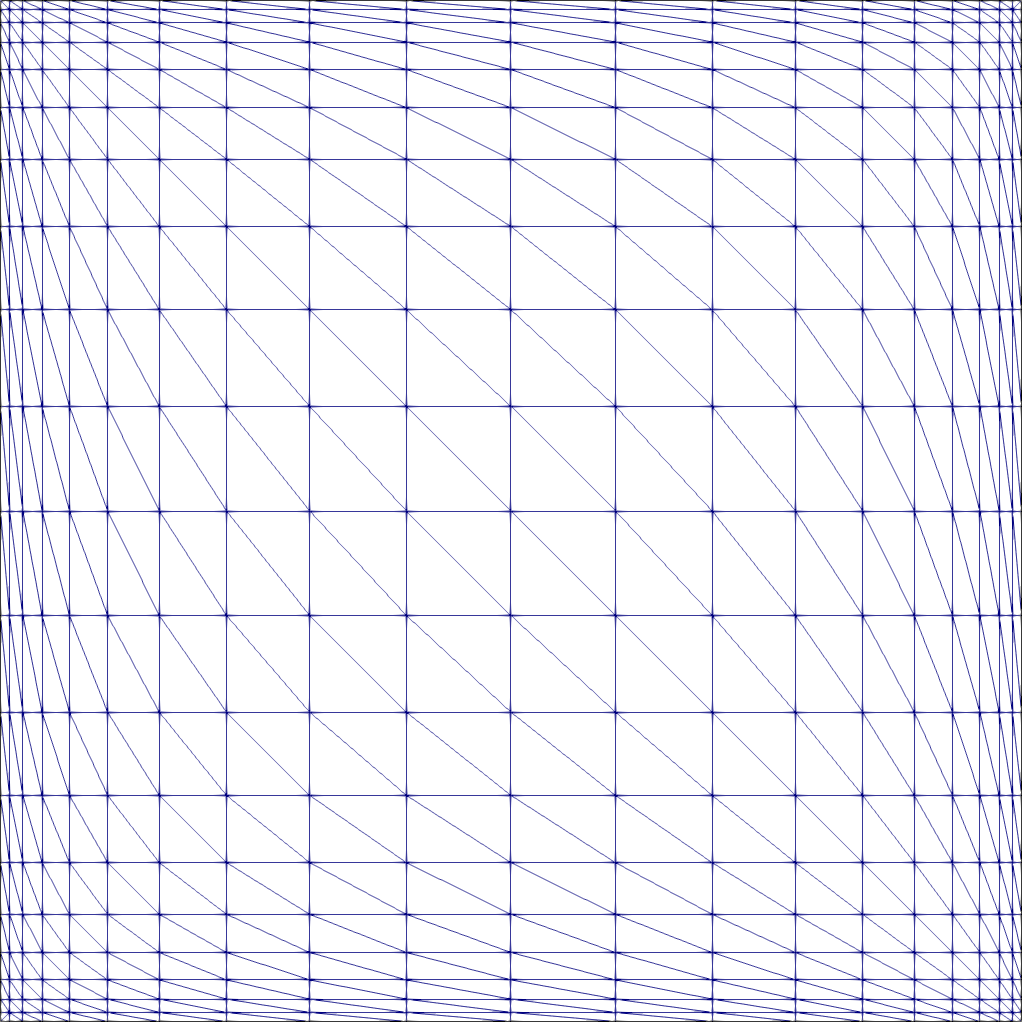}
    }\hspace*{3cm}
    \subfloat[Diffusivity coefficient $\kappa$\label{fig:kappa}]{
        \includegraphics[width=0.25\textwidth]{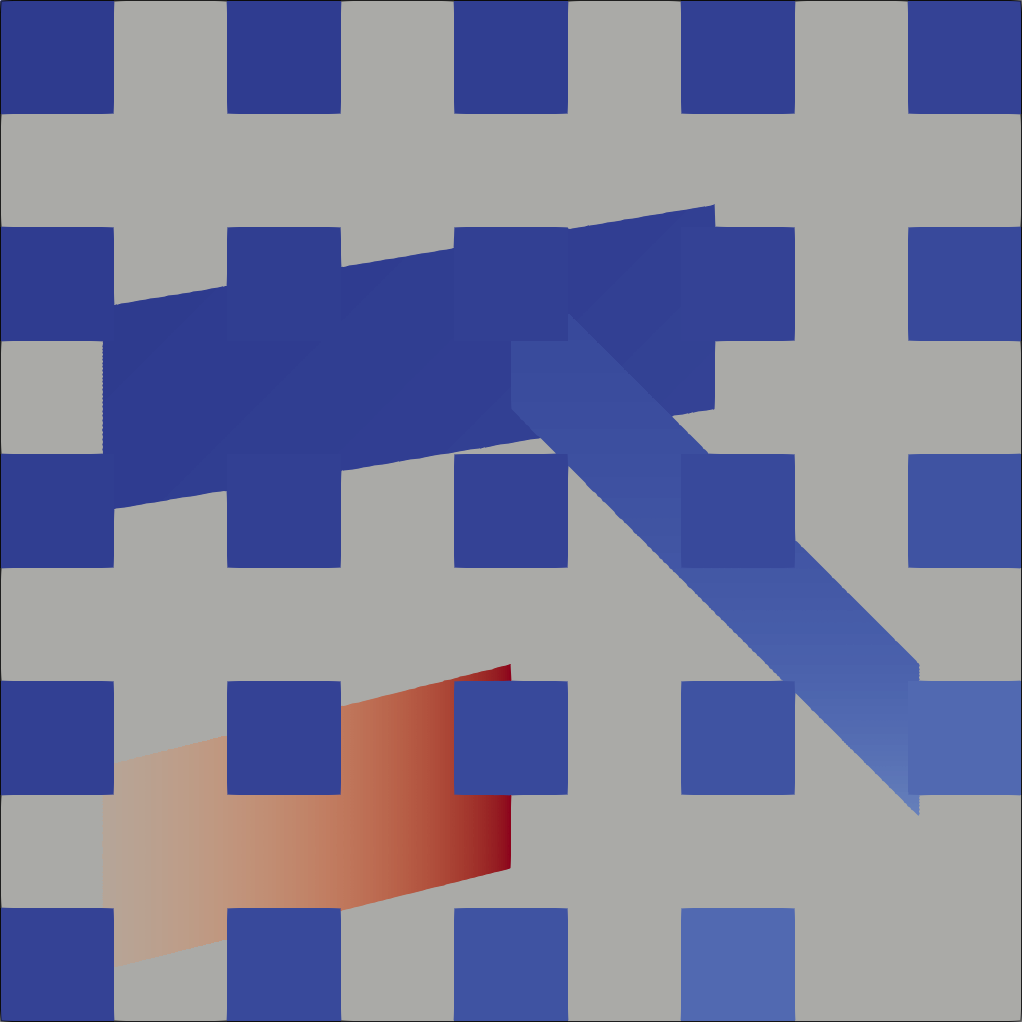}
        \raisebox{1cm}{\scalebox{0.7}{
            \begin{tikzpicture}[overlay]
	  \begin{axis}[height=4.0cm,
	  enlargelimits=false,
	  colorbar,
	  colormap name=paraview-heat,
	  point meta min=6e-2,
	  point meta max=2,
      hide axis,
      colorbar style={
          title={$\kappa$},
          scaled y ticks = false,
          extra y ticks={6e-2},
      }
	  ]
	  \end{axis}
	\end{tikzpicture}
}}
    } \\[1cm]
    \hfuzz=3.0pt
    \subfloat[$\nu=1$\label{fig:2d-0}]{
        \includegraphics[width=0.31\textwidth]{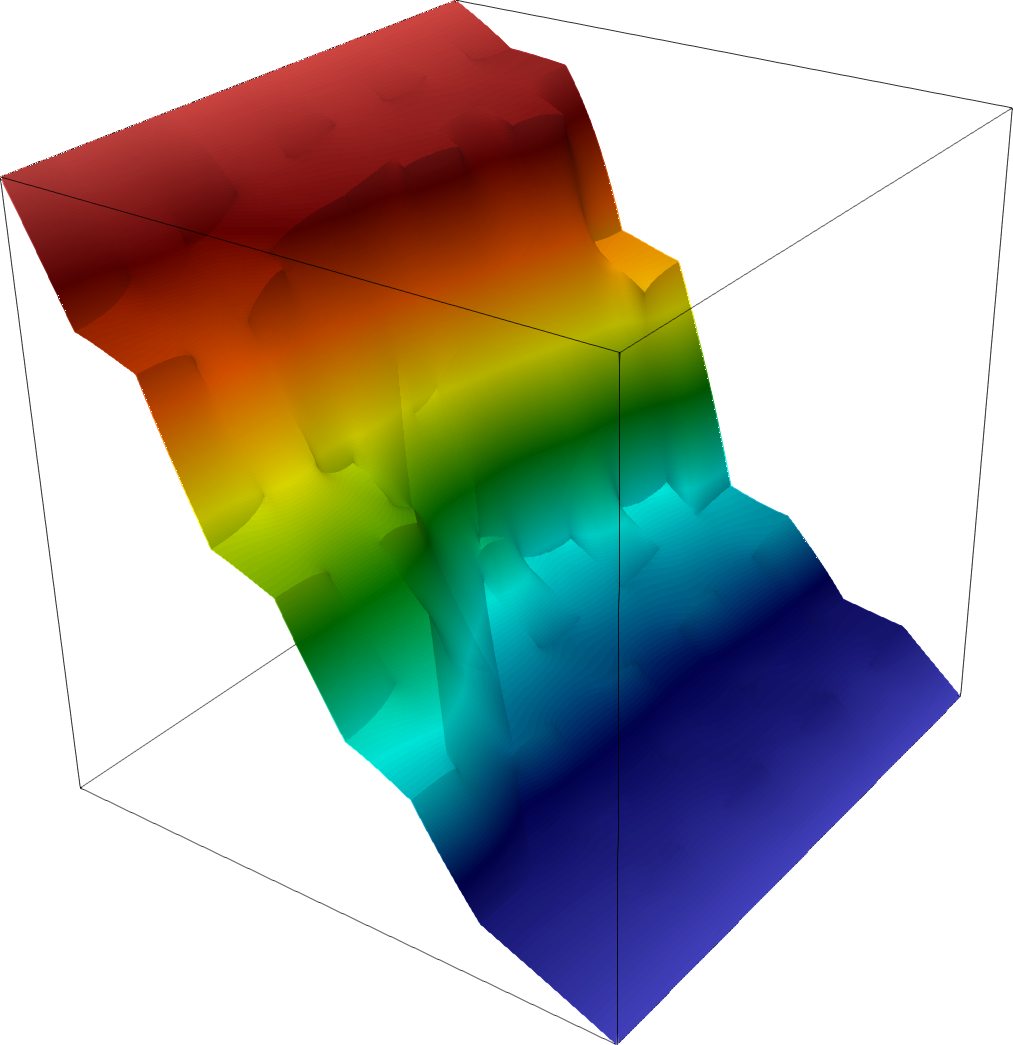}
    }
    \subfloat[$\nu=10^{-2}$\label{fig:2d-2}]{
        \scalebox{0.7}{
            \begin{tikzpicture}[overlay,yshift=7cm,xshift=2.0cm]
	  \begin{axis}[
	  colorbar,
	  colormap name=paraview,
	  point meta min=0,
	  point meta max=1,
      hide axis,height=4cm,
          colorbar horizontal,
      colorbar style={
          title={$u$},
      }
	  ]
	  \end{axis}
	\end{tikzpicture}
    }
        \includegraphics[width=0.31\textwidth]{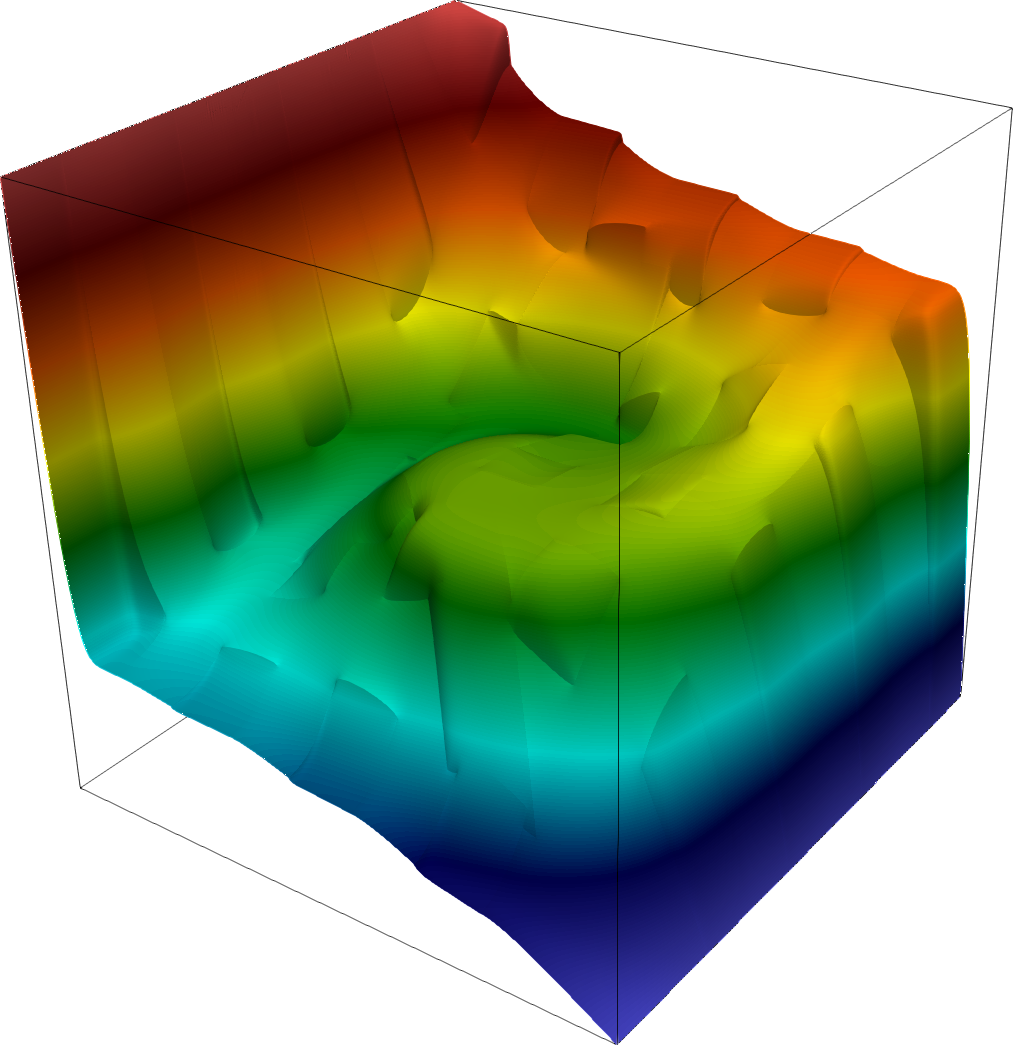}
    }
    \subfloat[$\nu=10^{-4}$\label{fig:2d-4}]{
        \includegraphics[width=0.31\textwidth]{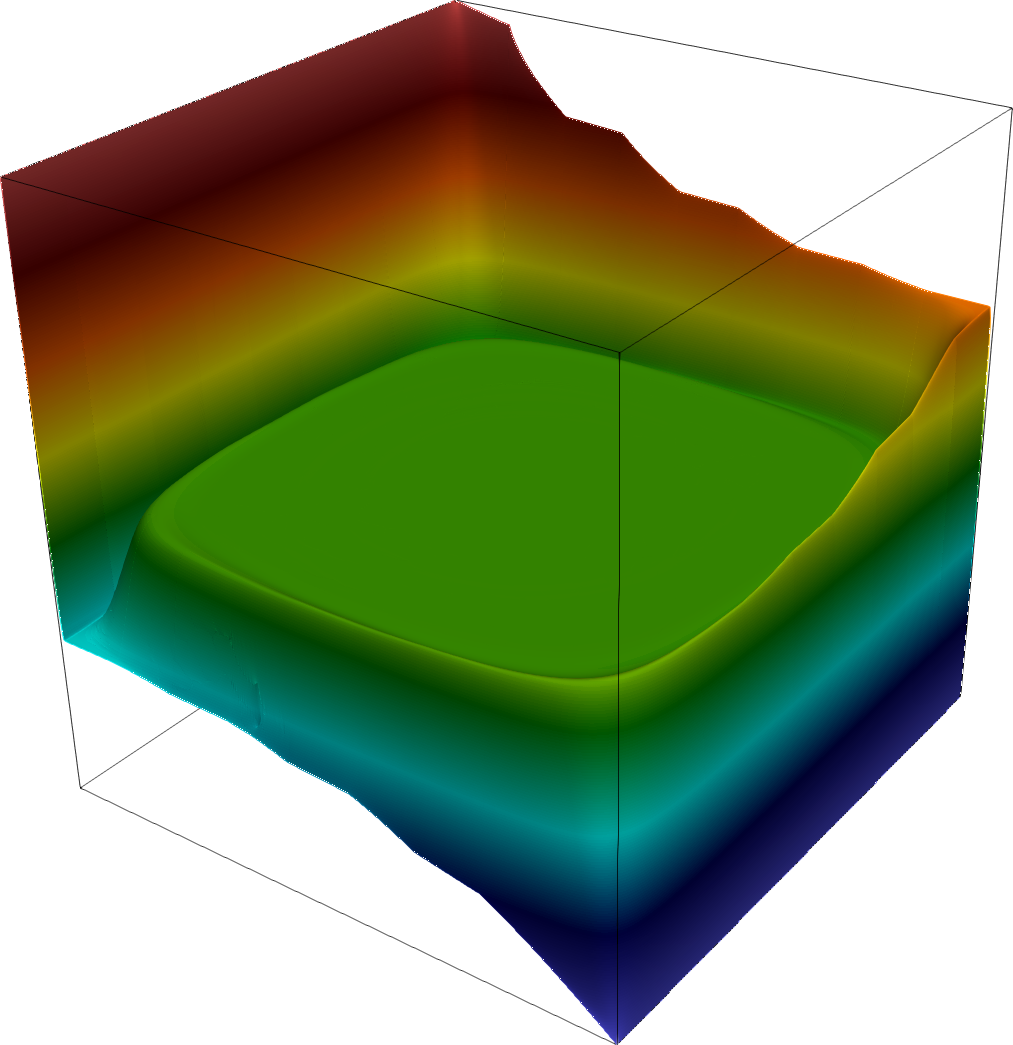}
    } \\
    \subfloat[$\nu=1$\label{fig:3d-0}]{
        \includegraphics[width=0.31\textwidth]{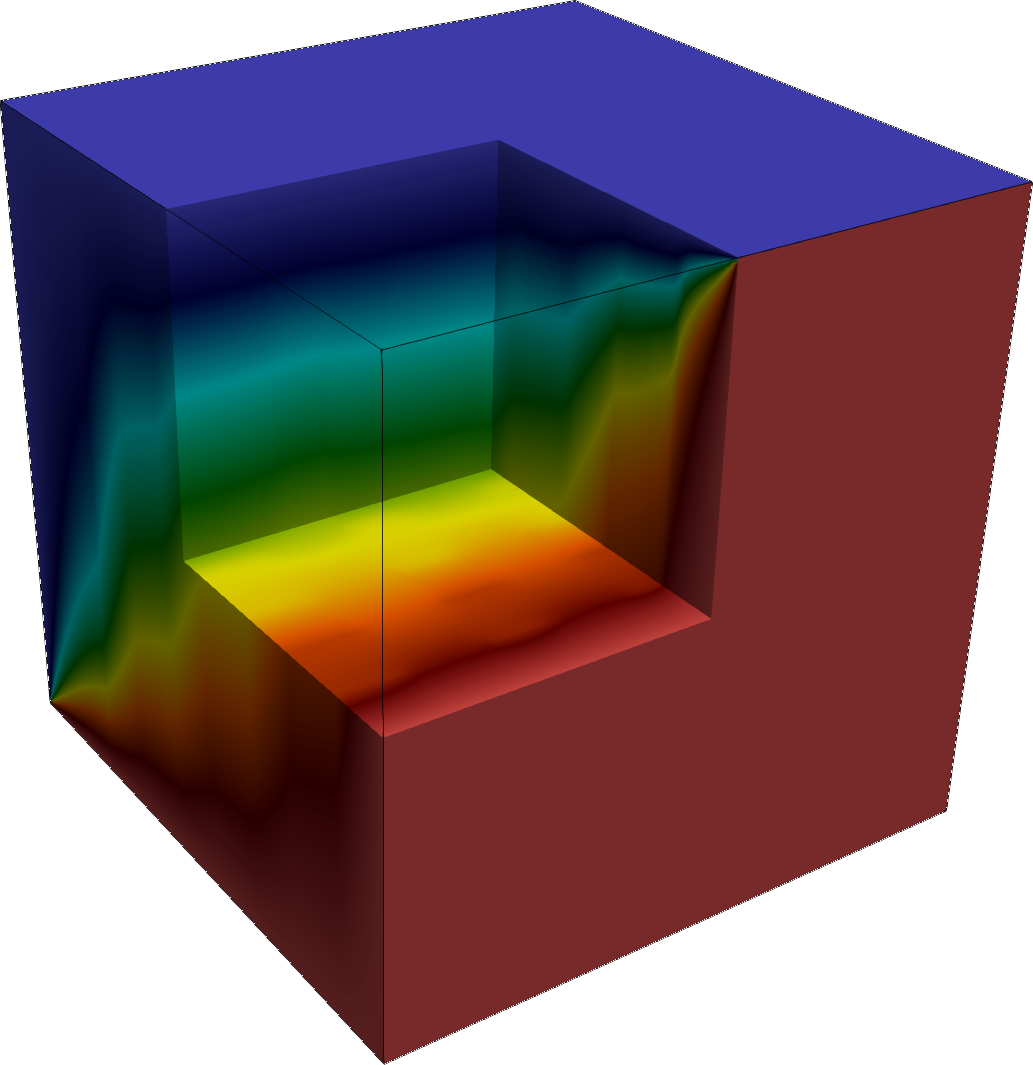}
    }
    \subfloat[$\nu=10^{-2}$\label{fig:3d-2}]{
        \includegraphics[width=0.31\textwidth]{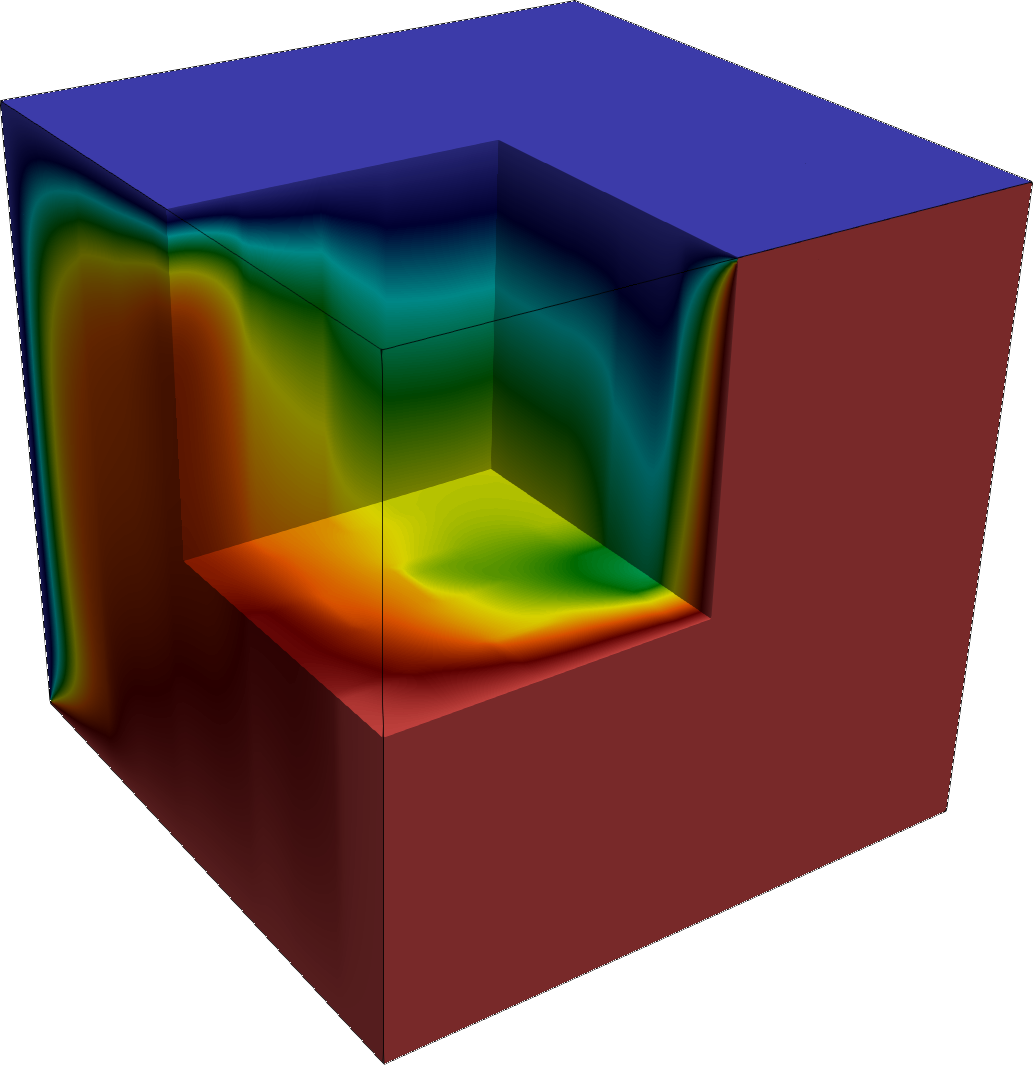}
    }
    \subfloat[$\nu=10^{-4}$\label{fig:3d-4}]{
        \includegraphics[width=0.31\textwidth]{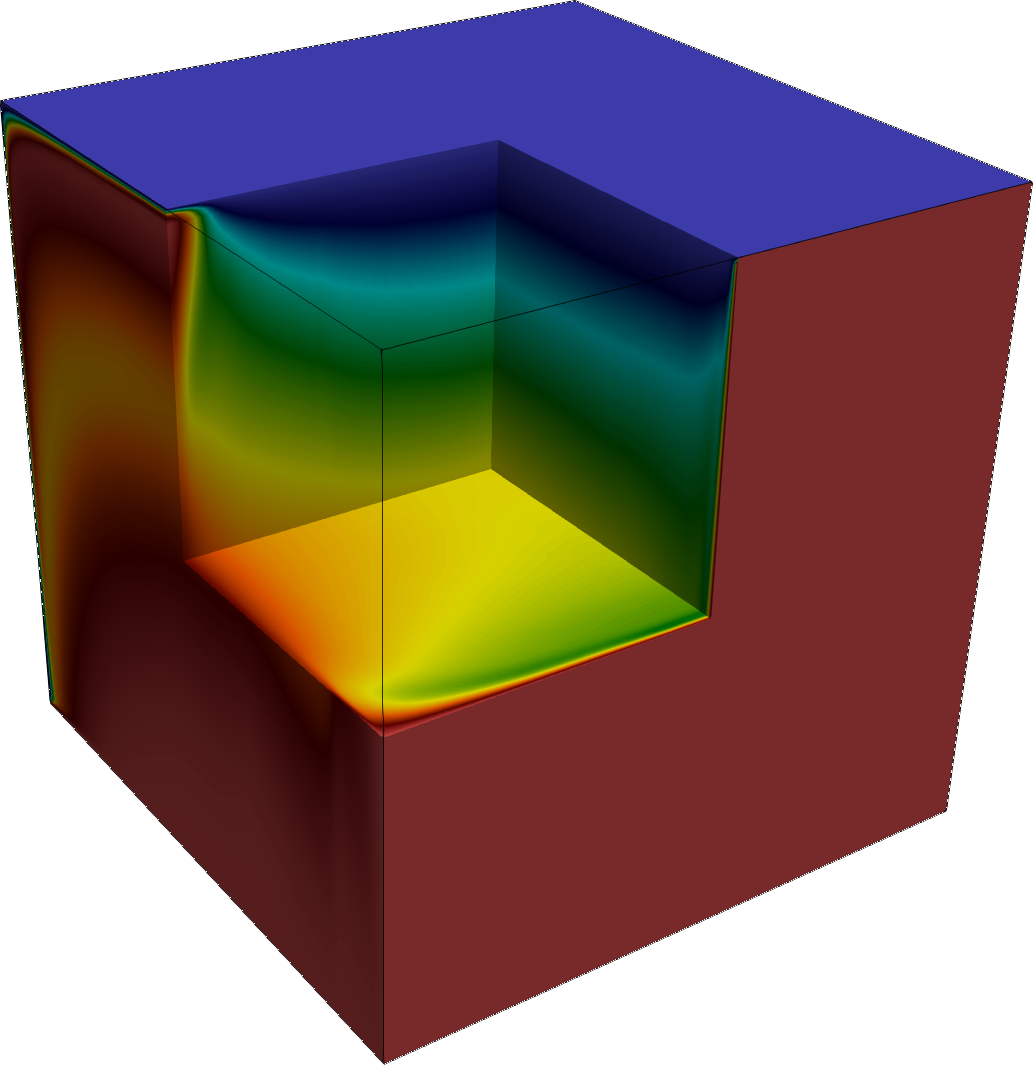}
    }
    \end{center}
    \hfuzz=0.0pt
\end{figure}
For comparison, GAMG and BoomerAMG iteration counts are also reported in \cref{tab:conv-diff-gamg} and \cref{tab:conv-diff-hypre}, respectively.
\begin{table}
    \centering
    \caption{Iteration counts of GAMG for solving the two- and three-dimensional convection-diffusion problem from~\cref{eq:conv-diff}. \dag~denotes either a failure to converge or a breakdown in GMRES. \label{tab:conv-diff-gamg}}
\begin{tabular}{c|c|cccccc}
  \hline
  \multirow{2}{*}{Dimension} &
  \multirow{2}{*}{$n$} &
  \multicolumn{5}{c}{$\nu$}
      \\ 
    & & 1 & $10^{-1}$ & $10^{-2}$ & $10^{-3}$  & $10^{-4}$ \\  \hline
    2 & \pgfmathprintnumber[precision=1]{6255001} & 42 & 48 & 88 & \dag & \dag \\\hline
    3 & \pgfmathprintnumber[precision=1]{8120601} & 40 & 38 & 65 & \dag & \dag \\\hline
\end{tabular}
\end{table}
\begin{table}
    \centering
    \caption{Iteration counts of BoomerAMG for solving the two- and three-dimensional convection-diffusion problem from~\cref{eq:conv-diff}. \dag~denotes either a failure to converge or a breakdown in GMRES. \label{tab:conv-diff-hypre}}
\begin{tabular}{c|c|cccccc}
  \hline
  \multirow{2}{*}{Dimension} &
  \multirow{2}{*}{$n$} &
  \multicolumn{5}{c}{$\nu$}
      \\ 
    & & 1 & $10^{-1}$ & $10^{-2}$ & $10^{-3}$  & $10^{-4}$ \\  \hline
    2 & \pgfmathprintnumber[precision=1]{6255001} & 50 & 49 & 19 & 7 & \dag \\\hline
    3 & \pgfmathprintnumber[precision=1]{8120601} & 12 & 9 & 7 & \dag & \dag \\\hline
\end{tabular}
\end{table}
\section{Conclusion}
\label{sec:conclusion}
We presented in this work a fully-algebraic two-level Schwarz preconditioner for large-scale sparse matrices.
The proposed preconditioner combines a classic one-level Schwarz preconditioner with a spectral coarse space.
The latter is constructed efficiently by solving concurrently in each subdomain a local generalized eigenvalue problem whose
pencil matrices are obtained algebraically and cheaply from the local coefficient matrix.
Convergence results were obtained for diagonally dominant HPD matrices.
The proposed preconditioner was compared to state-of-the-art multigrid preconditioners on a set of challenging matrices arising from 
a wide range of applications including a convection-dominant convection-diffusion equation.
The numerical results demonstrated the effectiveness and robustness of the proposed preconditioner especially for
highly non-symmetric matrices.

\appendix
\section*{Acknowledgments}
This work was granted access to the
GENCI-sponsored HPC resources of TGCC@CEA under allocation A0110607519. 

\newpage

\section{Code reproducibility}
\label{sec:code}
\lstset{emph={%
    CHKERRQ,ierr,%
    },emphstyle={\color{gray!90}},%
}%
\begin{lstlisting}[style=CStyle]
#include <petsc.h>

static char help[] = "Solves a linear system after having repartitioned its symmetric part.\n\n";

int main(int argc,char **args)
{
  Vec             b;
  Mat             A,perm;
  KSP             ksp;
  IS              is,rows;
  PetscBool       flg;
  PetscViewer     viewer;
  char            name[PETSC_MAX_PATH_LEN];
  MatPartitioning mpart;
  PetscErrorCode  ierr;

  ierr = PetscInitialize(&argc,&args,NULL,help);if (ierr) return ierr;
  ierr = MatCreate(PETSC_COMM_WORLD,&A);CHKERRQ(ierr);
  ierr = PetscOptionsGetString(NULL,NULL,"-mat_name",name,sizeof(name),&flg);CHKERRQ(ierr);
  if (!flg) SETERRQ(PETSC_COMM_WORLD,PETSC_ERR_USER,"Missing -mat_name");
  ierr = PetscViewerBinaryOpen(PETSC_COMM_WORLD,name,FILE_MODE_READ,&viewer);CHKERRQ(ierr);
  ierr = MatLoad(A,viewer);CHKERRQ(ierr);
  ierr = PetscViewerDestroy(&viewer);CHKERRQ(ierr);
  ierr = KSPCreate(PETSC_COMM_WORLD,&ksp);CHKERRQ(ierr);
  ierr = MatPartitioningCreate(PETSC_COMM_WORLD,&mpart);CHKERRQ(ierr);
  {
    Mat B,T;
    ierr = MatTranspose(A,MAT_INITIAL_MATRIX,&T);CHKERRQ(ierr);
    ierr = MatDuplicate(A,MAT_COPY_VALUES,&B);CHKERRQ(ierr);
    ierr = MatAXPY(B,1.0,T,DIFFERENT_NONZERO_PATTERN);CHKERRQ(ierr);
    ierr = MatPartitioningSetAdjacency(mpart,B);CHKERRQ(ierr); // partition A^T+A
    ierr = MatPartitioningSetFromOptions(mpart);CHKERRQ(ierr);
    ierr = MatPartitioningApply(mpart,&is);CHKERRQ(ierr);
    ierr = MatDestroy(&B);CHKERRQ(ierr);
    ierr = MatDestroy(&T);CHKERRQ(ierr);
  }
  ierr = MatPartitioningDestroy(&mpart);CHKERRQ(ierr);
  ierr = ISBuildTwoSided(is,NULL,&rows);CHKERRQ(ierr);
  ierr = ISDestroy(&is);CHKERRQ(ierr);
  ierr = MatCreateSubMatrix(A,rows,rows,MAT_INITIAL_MATRIX,&perm);CHKERRQ(ierr);
  ierr = ISDestroy(&rows);CHKERRQ(ierr);
  ierr = MatHeaderReplace(A,&perm);CHKERRQ(ierr); // only keep the permuted matrix
  ierr = KSPSetFromOptions(ksp);CHKERRQ(ierr); // parse command-line options
  ierr = KSPSetOperators(ksp,A,A);CHKERRQ(ierr);
  ierr = MatCreateVecs(A,NULL,&b);CHKERRQ(ierr);
  ierr = VecSetRandom(b,NULL);CHKERRQ(ierr); // random right-hand side
  ierr = KSPSolve(ksp,b,b);CHKERRQ(ierr);
  ierr = KSPDestroy(&ksp);CHKERRQ(ierr);
  ierr = VecDestroy(&b);CHKERRQ(ierr);
  ierr = MatDestroy(&A);CHKERRQ(ierr);
  ierr = PetscFinalize();
  return ierr;
}
\end{lstlisting}

\bibliographystyle{siamplain}
\bibliography{main}
\end{document}